\documentclass[11pt]{article}
\usepackage{amsmath, amssymb, amsthm}
\usepackage{verbatim}
\usepackage{multicol}
\usepackage{enumerate}
\usepackage{comment}
\usepackage[none]{hyphenat}
\usepackage{hyperref}
\hypersetup{
	colorlinks=true,
	linkcolor=blue,
	filecolor=magenta,
	urlcolor=cyan,
	citecolor=blue
}

\usepackage{pgf}
\usepackage{tikz}
\usetikzlibrary{math,calc,intersections}
\usetikzlibrary{positioning,arrows,shapes,decorations.markings,decorations.pathreplacing,matrix,patterns}
\tikzstyle{vertex}=[circle,draw=black,fill=black,inner sep=0,minimum size=3pt,text=white,font=\footnotesize]
\usepackage{cleveref}

\date{}
\title{\vspace{-0.9cm} The extremal number of tight cycles}
\author{
Benny Sudakov\thanks{ETH Zurich, \emph{e-mail}: \textbf{\{benjamin.sudakov,istvan.tomon\}@math.ethz.ch}. Research was supported by SNSF grant 200021-175573.}
\and
Istv\'an Tomon\footnotemark[1] \thanks{MIPT Moscow, Research partially supported by  the Ministry of Educational and Science of the Russian Federation in the framework of MegaGrant no 075-15-2019-1926.}
}

\theoremstyle{plain}
\newtheorem{theorem}{Theorem}[section]

\newtheorem{corollary}[theorem]{Corollary}
\newtheorem{claim}[theorem]{Claim}
\newtheorem{lemma}[theorem]{Lemma}
\newtheorem{conjecture}[theorem]{Conjecture}

\theoremstyle{definition}

\DeclareMathOperator{\ex}{ex}

\newcommand{\e}{{\bf e}}
\newcommand{\g}{{\bf g}}
\newcommand{\f}{{\bf f}}

\begin{document}

\sloppy

	\maketitle
	
\begin{abstract}
    A \emph{tight cycle} in an $r$-uniform hypergraph $\mathcal{H}$ is a sequence of $\ell\geq r+1$ vertices $x_1,\dots,x_{\ell}$ such that all $r$-tuples $\{x_{i},x_{i+1},\dots,x_{i+r-1}\}$ (with subscripts modulo $\ell$) are edges of $\mathcal{H}$.
    An old problem of V. S\'os, also posed independently by J. Verstra\"ete, asks for the maximum number of edges in an $r$-uniform hypergraph on $n$ vertices which has no tight cycle. Although this is a very basic question, until recently, no good upper bounds were known for this problem for $r\geq 3$. Here we prove that the answer is at most $n^{r-1+o(1)}$, which is tight up to the $o(1)$ error term. 
    Our proof is based on finding  robust expanders in the line graph of $\mathcal{H}$ together with certain
    density increment type arguments. 
\end{abstract}

	\section{Introduction}
	Given a (possibly infinite) family of $r$-uniform hypergraphs $\mathcal{F}$, the \emph{Tur\'an number} or \emph{extremal number} of $\mathcal{F}$, denoted by $\mbox{ex}(n,\mathcal{F})$, is the maximum number of edges in an $r$-uniform hypergraph on $n$ vertices which does not contain a copy of any member of $\mathcal{F}$. The study of the extremal numbers of graphs and hypergraphs is one of the central topics in discrete mathematics, which
	goes back more than hundred years to the works of Mantel \cite{Ma} in 1907 and Tur\'an \cite{Tu} in 1941. Instances of this problem also appear naturally in discrete geometry, additive number theory, probability, analysis, computer science and coding theory. For a general reference, we refer the reader to the surveys \cite{FS13,K11,MPS, Su}. Despite that this topic was extensively studied, there are still many natural families of graphs and hypergraphs whose Tur\'an number is not well understood. In this paper, we make a substantial progress on understanding the extremal number of such a family of hypergraphs.
	
	One of the most basic results in graph theory says that if $G$ is a graph on $n$ vertices which does not contain a cycle, then $G$ has at most $n-1$ edges, and this bound is best possible. While this simple fact has many different proofs, analogues of this result for uniform hypergraphs turn out to be more challenging. There are different notions of cycles in hypergraphs one can consider: loose cycles, Berge cycles and tight cycles (which all coincide for graphs), but in this paper we focus on tight cycles. If $r\geq 2$ and $\ell\geq r+1$, the \emph{tight cycle of length $\ell$} is the $r$-uniform hypergraph with vertices $x_{1},\dots,x_{\ell}$ and edges $\{x_{i},x_{i+1},\dots,x_{i+r-1}\}$ for $i=1,\dots,\ell$, where all indices are modulo~$\ell$.  There is a large literature on the extremal number of Berge and loose cycles, see e.g. \cite{BGy08,FJ14,Gy06,GyL12,JM18,KMV15}, but the corresponding questions for tight cycles turn out to be particularly difficult.
	
	Let $\mathcal{C}^{(r)}$ denote the family of $r$-uniform tight cycles. Let $S^{(r)}_n$ be the $r$-uniform hypergraph, whose edges are those $r$-element subsets of $[n]$ that contain 1.  Clearly, $S^{(r)}_n$ contains no tight cycle, and $|E(S^{(r)}_n)|=\binom{n-1}{r-1}$. S\'os, and independently Verstra\"ete (see, e.g., \cite{MPS, V16}) conjectured that $S^{(r)}_n$ is extremal for tight cycles, that is, $\mbox{ex}(n,\mathcal{C}^{(r)})=\binom{n-1}{r-1}$ for sufficiently large $n$. This conjecture was recently disproved by Huang and Ma \cite{HM19}, who showed that for every $r\geq 3$ there exists some constant $1<c=c(r)<2$ such that $\mbox{ex}(n,\mathcal{C}^{(r)})\geq c\binom{n-1}{r-1}$ for every sufficiently large $n$. On the other hand, it is widely believed that $\mbox{ex}(n,\mathcal{C}^{(r)})=O(n^{r-1})$. Nevertheless, there were no upper bounds known getting close to this conjecture. In the case $r=3$, an unpublished result of Verstra\"ete states that if $C^{(3)}_{24}$ is the 3-uniform tight cycle of length 24, then $\mbox{ex}(n,C^{(3)}_{24})=O(n^{5/2})$. For $r\geq 4$, the best upper bound we were aware of is $\mbox{ex}(n,\mathcal{C}^{(r)})=O(n^{r-2^{-r+1}})$, which comes from the observation that the complete $r$-partite $r$-uniform hypergraph with vertex classes of size 2, denoted by $K^{(r)}_{2, \dots,2}$, contains the tight cycle of length $2r$, and we have $\mbox{ex}(n,K^{(r)}_{2, \dots,2})=O(n^{r-2^{-r+1}})$ by a well known result of Erd\H{o}s \cite{E64}. In case one wants to find a tight cycle of linear size,  Allen, B\"ottcher, Cooley, and Mycroft \cite{ABCM17} proved that for every $0<\alpha,\delta<1$ and sufficiently large $n$ (with respect to $r$ and $\delta$), any $r$-uniform hypergraph with $n$ vertices and at least $(\alpha+\delta)\binom{n}{r}$ edges contains a tight cycle of length $\ell$ for any $\ell\leq\alpha n$, $r \mid \ell$. However, the proof of this result uses the Hypergraph Regularity Lemma, and thus not applicable in the setting of sparse hypergraphs with $o(n^r)$ edges. In this paper, we prove the first upper bound for containing a tight cycle which matches the lower bound up to an $n^{o(1)}$ factor.
	
	\begin{theorem}\label{thm:mainthm}
	 If $\mathcal{H}$ is an $r$-uniform hypergraph on $n$ vertices which does not contain a tight cycle, then $\mathcal{H}$ has at most $n^{r-1+o(1)}$ edges.
	\end{theorem}
	
	\noindent
	More precisely, our proof shows that that there exists $c=c(r)>0$ such that $\mathcal{H}$ has at most $n^{r-1}e^{c\sqrt{\log n}}$ edges.
	
	\section{Preliminaries}
	Let us start by describing the notation we are going to use, some of which might be slightly unconventional. As usual, $[r]$ denotes the set $\{1,\dots,r\}$, and $S_{r}$ denotes the set of all permutations of $[r]$. If $G$ is a graph and $X\subset V(G)$, the \emph{neighborhood} of $X$ is $N_{G}(X)=N(X)=\{y\in V(G)\setminus X:\exists x\in X, xy\in E(G)\}$. A \emph{tight path} in an $r$-uniform hypergraph $\mathcal{H}$ is a sequence of $\ell\geq r+1$ vertices $x_1,\dots,x_{\ell}$ such that $\{x_{i},\dots,x_{i+r-1}\}\in E(\mathcal{H})$ for $i=1,\dots,\ell-r+1$.
	
	Let $\mathcal{H}$ be an $r$-uniform hypergraph on $n$ vertices. By considering a random partition of $V(\mathcal{H})$ into $r$-parts, we can find an $r$-partite subgraph $\mathcal{H}'$ of $\mathcal{H}$ with at least $\frac{r!}{r^r}|E(\mathcal{H})|$ edges. Therefore, it is enough to verify Theorem \ref{thm:mainthm} for $r$-partite $r$-uniform hypergraphs.	Now suppose that $\mathcal{H}$ is an $r$-partite $r$-uniform hypergraph with vertex classes $A_1,\dots,A_r$, each of size at most $n$. Instead of working with hypergraphs, we find it more suitable to work with their \emph{line graphs}. Also, instead of viewing edges as $r$-element subsets of the vertices, it is better to work with $r$-tuples of vertices. This motivates the following definition.
	
	Say that a graph $G$ is an \emph{$r$-line-graph} if the vertex set of $G$ is a set of $r$-tuples $V\subset A_{1}\times \dots\times A_{r}$, and $x$ and $y$ are joined by an edge in $G$ if and only if $x$ and $y$ differ in exactly one coordinate. A \emph{subgraph} of an $r$-line-graph $G$ always refer to an induced subgraph of $G$, which is an $r$-line-graph as well by definition. Given an $r$-partite $r$-uniform hypergraph $\mathcal{H}$, we can naturally identify it with an $r$-line-graph. A \emph{tight cycle} in an $r$-line-graph refers to a sequence of vertices corresponding to the edges of a tight cycle in the associated hypergraph.
	
	Let $G$ be an $r$-line-graph and let $X\subset V(G)\subset A_{1}\times \dots\times A_{r}$. For $i\in [r]$ and $X\subset V(G)$, the \emph{$i$-boundary} of $X$, denoted by $\partial^{(i)}_{G}(X)=\partial^{(i)}(X)$ is the set of vertices $y\in V(G)$ for which $y$ has a neighbor in $X$ which differs from $y$ in the $i$-th coordinate. Also, the \emph{$i$-neighborhood} of $X$ is $N^{(i)}_{G}(X)=N^{(i)}(X)=\partial^{(i)}(X)\setminus X$.
	
	For $i\in [r]$, an \emph{$i$-block} of $G$ is a set of the form $x\cup N^{(i)}(x)$ for some $x\in V(G)$, and a \emph{block} is an $i$-block for some $i\in [r]$. Note that the block containing $x$ is the set of all vertices which only differ from $x$ in the $i$-th coordinate. Therefore, a block is a clique in $G$, and the $i$-blocks of $G$ form a partition of $V(G)$ for any $i\in [r]$. Let $p(G)$ denote the number of blocks of $G$, and define the \emph{density} of $G$ as $$\mbox{dens}(G)=\frac{\sum_{B}|B|}{p(G)}=\frac{r|V(G)|}{p(G)},$$
	where the sum iterates over all blocks $B$ of $G$.
	
	 The \emph{$i$-degree} of a vertex $x\in V(G)$ is $d^{(i)}_{G}(x)=d^{(i)}(x)=|N^{(i)}(x)|+1$, where we write $N^{(i)}(x)$ instead of $N^{(i)}(\{x\})$ (so $d_{i}(x)$ is the size of the $i$-block containing $x$). With slight abuse of notation, the minimum degree of $G$, denoted by $\delta(G)$, is the minimum of $d^{(i)}(x)$ over all $x\in V(G)$ and $i\in [r]$, which is the minimum size of a block in $G$.

	\subsection{An overview of the proof}
	
	Let $\mathcal{H}$ be an $r$-partite $r$-uniform hypergraph with vertex classes of size at most $N$, and let $G$ be the $r$-line-graph associated with $\mathcal{H}$. Let $n=dN^{r-1}$ be the number of edges of $\mathcal{H}$, then $|V(G)|=n$ and $p(G)\leq rN^{r-1}$. Therefore, 
	$$\mbox{dens}(G)=\frac{rn}{p(G)}\geq \frac{rdN^{r-1}}{rN^{r-1}}=d.$$
	Hence, Theorem \ref{thm:mainthm} is an immediate consequence of the following theorem.
	
	\begin{theorem}\label{thm:mainthm2}
		There exists $c=c(r)>0$ such that the following holds. If $G$ is an $r$-line-graph with $n$ vertices that does not contain a tight cycle, then $\mbox{dens}(G)\leq e^{c\sqrt{\log n}}$.
	\end{theorem}
	
	In the rest of the paper, we prove Theorem \ref{thm:mainthm2}. Let us briefly outline our proof strategy. Let $G$ be an $r$-line-graph of density $d$ such that $V(G)\subset A_1\times\dots\times A_r$. First, we show that $G$ contains a subgraph $H$ with minimum degree $\Omega(d)$ (as a reminder, here and everywhere else, minimum degree refers to our new definition of minimum degree) and good expansion properties, namely that every $X\subset V(H)$ of size at most $\frac{|V(H)|}{2}$ satisfies $|N(X)|\geq \lambda|X|$, where $\lambda=\Theta(\frac{1}{\log n})$.  Then, we show that $H$ is a robust expander, meaning that even if one removes of a few elements of $A_1\cup\dots\cup A_r$ (and thus deletes all the vertices containing a removed coordinate), $H$ still has good expansion properties. This can be found in Section \ref{sect:expansion}.
	
	Now fix an arbitrary permutation $\sigma\in S_r$. Let $x=(x_1,\dots,x_r),y=(y_1,\dots,y_r)\in V(G)$ such that $x_i\neq y_i$ for $i\in [r]$. Say that $y$ is a \emph{$\sigma$-neighbor} of $x$ if the following holds. Let $z_0=x$, and define $z_1,\dots,z_r\in A_1\times\dots\times A_r$ such that $z_i$ is the vector we get from $z_{i-1}$ after changing the $\sigma(i)$-th coordinate to $y_{\sigma(i)}$. Then $y=z_r$. If $z_1,\dots,z_{r-1}$ are all vertices of $G$, then say that $y$ is a $\sigma$-neighbor of $x$. If $X\subset V(G)$, the $\sigma$-boundary of $X$, denoted by  $\partial^{\sigma}_{G}(X)$, is the set of vertices $y$ which are the $\sigma$-neighbor of some $x\in X$. 
	
	This notion is useful for the following reason. Say that a sequence of vertices $v_1,\dots,v_k$ is a $\sigma$-path if $v_{i+1}$ is a $\sigma$-neighbor of $v_{i}$ for $i=1,\dots,k-1$, and no two vertices among $v_{1},\dots,v_{k}$ share a coordinate. Then a $\sigma$-path corresponds to a tight path in the associated hypergraph. Our goal is to show that the expansion of $H$ implies that for any pair of vertices $x,y\in V(H)$ not sharing a coordinate, there is a short $\sigma$-path $P$ starting with $x$ and ending with $y$. But then after removing the coordinates appearing in $P$ (except for the coordinates of $x$ and $y$), the remaining graph $H'$ still has good expansion properties. Therefore, we can find a $\sigma$-path $P'$ starting with $y$ and ending with $x$ in $H'$. But then $P\cup P'$ is a tight cycle in $H$, and we are done. Unfortunately, we are not quite able to show that there is a $\sigma$-path from any $x$ to any $y$, but we can prove that either this is the case, or we can find a small subgraph of $H$ with unusually high density. Then, we conclude the proof by a density increment type argument.
	
	The key observation that allows us to find short paths between pairs of vertices is that if $H$ has good expansion properties, then the $\sigma$-boundaries also expand, that is, $|\partial^{\sigma}_{G}(X)|\geq (1+\lambda')|X|$ for every $|X|\leq \frac{1}{2}|V(H)|$ for some $\lambda'=\Theta(\lambda)$. This implies that given $x,y\in V(H)$, we can reach a large proportion of the  vertices of $H$ from $x$ by a short $\sigma$-path. Also, if $\tau$ is the reverse of $\sigma$, we can reach a large proportion of the vertices of $H$ by a $\tau$-path starting with $y$. See Section \ref{sect:sigma} for details. Then, it remains to find some $z\in V(G)$ such that $z$ can be reached from $x$ by a $\sigma$-path $P_x$, $z$ can be reached by a $\tau$-path $P_y$ from $y$, and no vertex $u\in P_x\setminus\{z\}$ and $v\in P_y\setminus\{z\}$ share a coordinate.  
	
	\section{Expansion}\label{sect:expansion}
	
	Let us start with a simple, but very useful lemma about finding subgraphs of large minimum degree in $r$-line-graphs.
	
	\begin{lemma}\label{lemma:mindeg}
		If $G$ is an $r$-line-graph of density $d$, then $G$ contains a subgraph $H$ such that $\mbox{dens}(H)\geq d$ and $\delta(H)\geq\frac{d}{r}$.
	\end{lemma}
	
	\begin{proof}
		Repeat the following operation: if there exist $i\in [r]$ and $x\in V(G)$ such that $d_{i}(x)<\frac{d}{r}$, then delete the block $B$ containing $x$. We show that if the density of $G$ is at least $d$, then this operation increases the density. Indeed, if the resulting graph is $G'$, then $$\mbox{dens}(G')\geq\frac{r|V(G)|-r|B|}{p(G)-1}>\frac{dp(G)-d}{p(G)-1}=d.$$
		
		This implies that after repeating the operation described above a finite number of times, we end up with a nonempty graph $H$ with the desired properties.
	\end{proof}
	
		 If $G$ is a graph and $\lambda>0$, we say that $G$ is a \emph{$\lambda$-expander} if for every $X\subset V(G)$ satisfying $|X|\leq \frac{1}{2}|V(G)|$ we have $|N(X)|\geq \lambda |X|$. Note that having expansion for every set of size at most $\frac{1}{2}|V(G)|$ automatically implies the expansion of larger sets as well, as we show in the next claim.
	
	\begin{claim}\label{claim:eps_exp}
		Let $0<\lambda<1$, $\epsilon>0$ and let $G$ be a $\lambda$-expander. If $X\subset V(G)$ such that 
		$|X|\leq (1-\epsilon)|V(G)|$, then $|N(X)|\geq \frac{\lambda\epsilon}{2}|X|$.
	\end{claim}
	
	\begin{proof}
		If $|X|\leq \frac{|V(G)|}{2}$, then this follows from the definition. Suppose that $|X|\geq \frac{|V(G)}{2}$ and that $|N(X)|< \frac{\lambda\epsilon}{2} |X|$. Let $Y=V(G)\setminus (X\cup N(X))$. Then $|Y|\leq \frac{|V(G)|}{2}$ so $|N(Y)|\geq \lambda|Y|$. But $|Y| = |V(G)|-|X|-|N(X)|\geq |V(G)|(\epsilon-\frac{\epsilon\lambda}{2})$ and $N(Y)\subset N(X)$, so 
		$$|N(X)|\geq |N(Y)|\geq  \lambda |V(G)|\left(\epsilon-\frac{\epsilon\lambda}{2}\right)\geq \frac{\epsilon\lambda}{2}|V(G)|,$$
		contradiction.
	\end{proof}
	
	Say that an $r$-line-graph $G$ is a \emph{$(\lambda,d)$-expander} if $G$ is $\lambda$-expander and $\delta(G)\geq d$. A result of Shapira and Sudakov \cite{SS15} tells us that every graph on $n$ vertices contains a $\lambda$-expander subgraph of roughly the same density, where $\lambda=\Theta(\frac{1}{\log n})$ (in their case, density refers to the usual notion of edge density; also, their notion of expansion is stronger). We use their approach to show that an $r$-line-graph $G$ also contains a $(\lambda,d)$-expander subgraph of roughly the same density, where $d=\Omega(\mbox{dens}(G))$.
	
	\begin{lemma}\label{lemma:expander}
		Let $G$ be an $r$-line-graph on $n$ vertices of density $d$, and let $0<\lambda\leq \frac{1}{2\log_{2} n}$. Then $G$ contains a subgraph $H$ of density at least $d(1-\lambda\log_{2} n)$ such that $H$ is $(\lambda,\frac{d}{2r})$-expander.
	\end{lemma}
	
	\begin{proof}
		First, we show that if $G'$ is an $r$-line-graph of density $d'$ that is \emph{not} a $\lambda$-expander, then there exists $U\subset V(G')$, $U\neq V(G')$ such that either 
		\begin{enumerate}
			\item $|U|\leq \frac{1}{2}|V(G')|$ and $\mbox{dens}(G'[U])\geq d'(1-\lambda)$, or
			\item $\mbox{dens}(G'[U])\geq d'$.
		\end{enumerate}
		Indeed, if $G'$ is not a $\lambda$-expander, then there exists $W\subset V(G')$ such that $|W|\leq \frac{1}{2}|V(G')|$ and $|N(W)|\leq \lambda|W|$. We show that either $U_{1}=W$ satisfies 1., or $U_{2}=V(G')\setminus (N(W)\cup W)$ satisfies 2.. Suppose this is not true. Let $p_{j}=p(G'[U_{j}])$ for $j=1,2$. Note that $p(G')\geq p_{1}+p_{2}$. But then we can write
		\begin{align*}
		r|V(G')|&=r(|U_1|+|N(U_{1})|+|U_{2}|)\leq r|U_{1}|(1+\lambda)+r|U_{2}|\\
		&<d'(1-\lambda)(1+\lambda)p_{1}+d'p_{2}\leq d'(p_{1}+p_{2})\leq r|V(G')|,
		\end{align*}
		contradiction.
		
		By applying Lemma \ref{lemma:mindeg}, we can also conclude that there exists $U\subset V(G')$, $U\neq V(G')$ such that either 
		\begin{itemize}
			\item[1.$^{*}$] $|U|\leq \frac{1}{2}|V(G')|$, $\mbox{dens}(G'[U]))\geq d'(1-\lambda)$ and $\delta(G'[U])\geq\frac{d'(1-\lambda)}{r}$, or
			\item[2.$^{*}$] $d(G'[U])\geq d'$ and $\delta(G'[U])\geq \frac{d'}{r}$.
		\end{itemize}
		Starting with an $r$-line-graph $G$ of density $d$, replace $G$ with one of its subgraphs of minimum degree at least $d/r$ having density at least $d$. If current $G$ is not a $\lambda$-expander, we can find $U\subset V(G)$ satisfying 1.$^{*}$ or 2.$^{*}$. Replace $G$ with $G[U]$, and repeat the previous step until $G$ is a $\lambda$-expander or its density is less than $d/2$. The process must stop since we are deleting at least one vertex at every step. Let $H$ be the final graph and let $\ell$ be the number of steps of kind  1.$^{*}$ that we made. Then  $|V(H)|\leq |V(G)|2^{-\ell}$ and therefore $\ell\leq \log_{2}n$. This implies that
		$\mbox{dens}(H)\geq d(1-\lambda)^{\ell}\geq d(1-\lambda\log_{2}n)\geq d/2$, $\delta(H)\geq \frac{d(1-\lambda\log_{2}n)}{r}\geq \frac{d}{2r}$ and $H$ is $\lambda$-expander.
	\end{proof}
	
	From this, we can immediately conclude that we can cover almost every vertex of an $r$-line-graph $G$ with disjoint expander subgraphs.
	
	\begin{corollary}\label{lemma:expander_covering}
		Let $\epsilon>0$. Let $G$ be an $r$-line-graph on $n$ vertices of density at least $d$, and let $\lambda\leq \frac{1}{2\log_2 n}$. Then $G$ contains vertex disjoint subgraphs $G_{1},\dots,G_{k}$ such that $G_{i}$ is a $(\lambda,\frac{\epsilon d}{2r})$-expander for $i\in [k]$, and $|\bigcup_{i=1}^{k} V(G_{i})|\geq (1-\epsilon)n$. 
	\end{corollary}
	
	\begin{proof}
	We will greedily find the subgraphs $G_1,\dots,G_k$ as follows. Suppose that we have already found $G_{1},\dots,G_{j}$, and let $H$ be the subgraph of $G$ induced on $V(G)\setminus \bigcup_{i=1}^{j}V(G_{i})$. If $|V(H)|\leq \epsilon n$, then stop, otherwise define $G_{j+1}$ as follows. The number of blocks of $H$ is at most $p(G)$, so the density of $H$ is at least $\frac{r\epsilon n}{p(G)}=\epsilon d.$ But then $H$ contains a $(\lambda,\frac{\epsilon d}{2r})$-expander subgraph by Lemma \ref{lemma:expander}, let this subgraph be $G_{j+1}$.
	\end{proof}
	

	Let $G$ be an $r$-line-graph with $V(G)\subset A_1\times \dots\times A_r$. If $u\in A_1\cup \dots\cup A_r$, \emph{the deletion of $u$ from $G$} means that we remove all vertices of $G$ with one coordinate equal to $u$. Next, we show that our notion of $(\lambda,d)$-expansion is robust, which means that after the deletion of a few coordinates of a good expander, the resulting graph is still a good expander.
	
	\begin{lemma}\label{lemma:expander_robust}
		Let $u,d$ be positive integers.  Let $G$ be an $r$-line-graph on $n$ vertices with $V(G)\subset A_1\times \dots\times A_r$ such that the minimum degree of $G$ is $\delta$.  Let $H$ be the subgraph of $G$ we get after deleting at most $u$ elements of $A_1\cup \dots\cup A_r$ from $G$. Then $H$ is an $r$-line-graph of minimum degree at least $\delta-u$ on at least $(1-\frac{u}{\delta})n$ vertices.
		
		If in addition $G$ is $\lambda$-expander, $\lambda\leq 1$ and $u\leq \frac{\lambda \delta}{4r}$, then $H$ is a $(\frac{\lambda}{2},\frac{d}{2})$-expander.
	\end{lemma}
	
	\begin{proof}
		If $x\in V(H)$, then at most $u$ neighbors of $x$ are deleted, so it is clear that the minimum degree of $H$ is at least $\delta-u$. Let $U$ be the set of deleted elements and let $U_i=A_i$ for $i\in [r]$. The number of vertices $(a_1,\dots,a_r)\in V(G)$ such that $a_{i}\in U_i$ for some $i\in [r]$ is at most $\frac{|U_{i}|}{\delta}n$, as each $i$-block of $G$ contains at least $\delta$ vertices of which at most $|U_i|$ has its $i$-th coordinate in $U_i$. Therefore, the total number of deleted vertices is at most $\frac{|U_1|+\dots+|U_r|}{\delta}n\leq\frac{u}{\delta}n.$
		
		It remains to show that $H$ is a $\frac{\lambda}{2}$-expander. Let $X\subset V(H)$ such that $|X|\leq \frac{1}{2}|V(H)|$. As $G$ is $\lambda$-expander, we have $|N_{G}(X)|\geq \lambda |X|$. Let $i\in [r]$, and let $\mathcal{B}$ be the family of $i$-blocks of $G$ having a nonempty intersection with $X$. Then the $i$-blocks of $H$ intersecting $X$ are $V(H)\cap B$ for $B\in \mathcal{B}$. But here, we have
		$$|V(H)\cap B|\geq |B|-u\geq |B|\left(1-\frac{u}{\delta}\right)\geq |B|\left(1-\frac{\lambda}{4r}\right).$$
		Note that 
		$$|N_{G}^{(i)}(X)|+|X|=\sum_{B\in \mathcal{B}}|B|,$$
		so
		$$|N_{H}^{(i)}(X)|+|X|=\sum_{B\in\mathcal{B}}|B\cap V(H)|\geq \left(1-\frac{\lambda}{4r}\right)\sum_{B\in\mathcal{B}}|B|= \left(1-\frac{\lambda}{4r}\right)(|N_G^{(i)}(X)|+|X|).$$
		From this, we get
		\begin{equation}\label{equ:robust}
		|N_{H}^{(i)}(X)|\geq |N_{G}^{(i)}(X)|-\frac{\lambda}{4r}(|X|+|N_{G}^{(i)}(X)|) .
		\end{equation}
		Consider two cases.
		\begin{itemize}
		    \item[Case 1.] There exists $i\in [r]$ such that $|N_{G}^{(i)}(X)|\geq |X|$. In this case (\ref{equ:robust}) implies that $$|N_{H}(X)|\geq |N_{H}^{(i)}(X)|\geq \left(1-\frac{\lambda}{2r}\right)|N^{(i)}_{G}(X)|>\frac{1}{2}|X|.$$
		    \item[Case 2.] For every $i\in [r]$, we have $|N_{G}^{(i)}(X)|< |X|$. Then (\ref{equ:robust}) implies that $|N_{H}^{(i)}(X)|\geq |N^{(i)}_G(X)|-\frac{\lambda}{2r}|X|$. But then 
		    $$|N_{H}(X)|=\left|\bigcup_{i\in [r]} N^{(i)}_H(X)\right|\geq \left|\bigcup_{i\in [r]} N^{(i)}_G(X)\right|-r\cdot\frac{\lambda}{2r}|X|=|N_G(X)|-\frac{\lambda}{2}|X|\geq \frac{\lambda}{2}|X|.$$
		\end{itemize}
	 Therefore, $H$ is a $\frac{\lambda}{2}$-expander.
	\end{proof}
	
	\section{$\sigma$-expansion}\label{sect:sigma}
	
	Let $G$ be an $r$-line-graph and let $X\subset V(G)$. Given a permutation $\sigma\in S_r$, the \emph{$\sigma$-boundary }of $X$ is defined as $$\partial_{G}^{\sigma}(X)=\partial^{\sigma}(X)=\partial^{(\sigma(r))}(\dots\partial^{(\sigma(2))}(\partial^{(\sigma(1))}(X))\dots).$$ If $x\in V(G)$ and $y\in \partial^{\sigma}(x)$, say that \emph{$y$ is a $\sigma$-neighbor of $x$ in $G$}. Note that being a $\sigma$-neighbour is not necessarily a symmetric relation.
	
	In this section, we show that if $G$ has good expansion properties, then the $\sigma$-boundaries also expand. We prove even more: even if one deletes a few $\sigma$-neighbours of every $x\in V(G)$ arbitrarily, the $\sigma$-boundaries still expand. This motivates the following definition.
	
	Suppose that $V(G)\subset A_{1}\times\dots\times A_{r}$. For every $x\in V(G)$, let $F(x)\subset A_1\cup \dots\cup A_r$ be a (possibly empty) set of forbidden coordinates, and let $\mathcal{F}=(F(x))_{x\in V(G)}$. Given $X\subset V(G)$, define $\partial^{\sigma}_G(X,\mathcal{F})=\partial^{\sigma}(X,\mathcal{F})$ as the set of vertices $y$ for which there exists $x\in X$ such that $y$ is a $\sigma$-neighbor of $x$, and no coordinate of $y$ is in $F(x)$.
	
	\begin{lemma}\label{lemma:sigma_expansion}
		Let $\sigma\in S_r$, let $0<\epsilon,\lambda<1$, and let $d,u$ be a positive integers such that $100r^2 u\leq \epsilon d\lambda$. Let $G$ be an $r$-line-graph such that $G$ is a $(\lambda,d)$-expander, and let $\mathcal{F}=(F(x))_{x\subset V(G)}$ such that $|F(x)|\leq u$ for every $x\in V(G)$. Then for every $X\subset V(G)$ satisfying $|X|\leq (1-\epsilon)|V(G)|$, we have $$|\partial^{\sigma}(X,\mathcal{F})|\geq \left(1+\frac{\epsilon\lambda}{4r}\right)|X|.$$
	\end{lemma}
	
	\begin{proof}
		Write $t=\frac{u}{d}$, then $t\leq \frac{\epsilon \lambda}{100r^{2}}$. Let $i\in [r]$ and $X\subset V(G)$. We start the proof with two simple claims. Let $\mathcal{F}'=(F'(x))_{x\in V(G)}$ be some family of sets of forbidden coordinates, and define $\partial^{(i)}(X,\mathcal{F}')$ as the set of vertices $y$ which have an $i$-neighbor $x\in X$ such that the $i$-th coordinate of $y$ is not in $F'(x)$. Suppose that $|F'(x)|\leq u$ for every $x\in V(G)$.
		
		\begin{claim}\label{claim:expansion}
			$$|\partial^{(i)}(X,\mathcal{F}')|\geq (1-4t)|X\cup N^{(i)}(X)|\geq (1-4t)|X|.$$
		\end{claim}
		\begin{proof}
			It is enough to show the first inequality. Let $\mathcal{B}$ be the set of $i$-blocks in $G$ having a nonempty intersection with $X$. Clearly, $\partial^{(i)}(X,\mathcal{F}')$ is the disjoint union of the sets $\partial^{(i)}(B\cap X,\mathcal{F}')$ for $B\in\mathcal{B}$, and $X\cup N^{(i)}(X)$ is the disjoint union of the sets $B\in \mathcal{B}$. Therefore, it is enough to show that $|\partial^{(i)}(B\cap X,\mathcal{F}')|\geq (1-4t)|B|$. But if $x\in B\cap X$, then 
			$$|\partial^{(i)}(B\cap X,\mathcal{F}')|\geq |B|-1-|F(x)|\geq |B|-1-dt\geq |B|(1-4t),$$
			where the third inequality holds noting that $|B|\geq d$.
		\end{proof}
		\begin{claim}\label{claim:nonexpansion}
			If $|X\cup N^{(i)}(X)|\leq 2|X|$, then $|X\cap \partial^{(i)}(X,\mathcal{F}')|\geq (1-4t)|X|$.
		\end{claim}
		\begin{proof}
			As before, let $\mathcal{B}$ be the set of $i$-blocks in $G$ having a nonempty intersection with $X$. Note that $$ |X\cup N^{(i)}(X)|=\sum_{B\in\mathcal{B}}|B|\geq d|\mathcal{B}|,$$
			so we get $|\mathcal{B}|\leq \frac{2|X|}{d}$.
			
			Let $B\in \mathcal{B}$ and $x\in B\cap X$. Clearly, 
			$$|\partial^{(i)}(B\cap X,\mathcal{F}')\cap X|\geq |B\cap X|-1-|F(x)|\geq |B\cap X|-1-dt.$$
			Therefore, we have
			$$|\partial^{(i)}(X,\mathcal{F}')\cap X|\geq \sum_{B\in\mathcal{B}}(|B\cap X|-1-dt)=|X|-2dt|\mathcal{B}|\geq |X|(1-4t),$$
			where the second inequality holds by $|\mathcal{B}|\leq \frac{2|X|}{d}$.
		\end{proof}
		Without loss of generality, suppose that $\sigma$ is the permutation $12\dots r$, and let $X\subset V(G)$ such that $|X|\leq (1-\epsilon) |V(G)|$. Let $X_{0}=X$, $\mathcal{F}_{0}=\mathcal{F}$, and in what follows we define $X_{1},\dots,X_{r}$ and $\mathcal{F}_1,\dots,\mathcal{F}_r$. Suppose that $X_{i},\mathcal{F}_{i}=(F_{i}(x))_{x\in V(G)}$ is already defined for some $i\in \{0,\dots,r-1\}$, then define $X_{i+1},\mathcal{F}_{i+1}$ as follows. Consider two cases. 
		
		First, suppose that  $|N^{(i+1)}(X_{i})|\geq \frac{\epsilon\lambda}{3r}|X_{i}|$, and in this case say that \emph{expansion happened}. Let $X_{i+1}=\partial^{(i+1)}(X_{i},\mathcal{F}_i)$. Then $|X_{i+1}|\geq (1-4t)(1+\frac{\epsilon\lambda}{3r})|X_{i}|$ by Claim \ref{claim:expansion}. On the other hand, if $|N^{(i+1)}(X_{i})|< \frac{\epsilon\lambda}{3r}$, then let $X_{i+1}=X_{i}\cap \partial^{(i+1)}(X_{i},\mathcal{F}_i)$. Then $|X_{i+1}|\geq (1-4t)|X_i|$ by Claim \ref{claim:nonexpansion}. In each case, for each $y\in X_{i+1}$, let $F_{i+1}(y)=F_{i}(x)$, where $x\in X_i$ is an arbitrary $(i+1)$-neighbor of $y$ for which no coordinate of $y$ is in $F_{i}(x)$, and set $F_{i+1}(y)=\emptyset$ for every $y\in V(G)\setminus X_{i+1}$.
		
		Note that $X_{r}\subset \partial^{\sigma}(X,\mathcal{F})$, so it is enough to show that $|X_{r}|\geq (1+\frac{\epsilon\lambda}{4r})|X|$. Observe that if expansion happened even once, then we have 
		$$|X_{r}|\geq (1-4t)^{r}\left(1+\frac{\epsilon\lambda}{3r}\right)|X|\geq (1-4rt) \left(1+\frac{\epsilon\lambda}{3r}\right)|X| \geq
		\left(1+\frac{\epsilon\lambda}{4r}\right)|X|,$$
		where the second inequality holds as $t\leq \frac{\epsilon \lambda}{100r^{2}}$. Therefore, it remains show that expansion must have happened. Suppose otherwise, then $X_{r}\subset X_{r-1}\subset\dots\subset X_{0}$ and $|X_{r}|\geq (1-4t)^{r}|X|\geq (1-4rt)|X|.$ Observe that for $i\in \{0\dots,r-1\}$, we have $X_{r}\cup N^{(i+1)}(X_{r})\subset X_{i}\cup N^{(i+1)}(X_{i})$. Therefore,
		$$|N^{(i+1)}(X_{r})|\leq |X_{i}|+|N^{(i+1)}(X_{i})|-|X_{r}|\leq |X|+\frac{\epsilon\lambda}{3r}|X|-(1-4rt)|X|=\left(4rt+\frac{\epsilon\lambda}{3r}\right)|X|, $$
		where the second inequality holds as expansion did not happen. But then
		$$|N(X_{r})|\leq\sum_{i=1}^{r}|N^{(i)}(X_{r})|\leq \left(4r^2t+\frac{\epsilon\lambda}{3}\right)|X|<\frac{\epsilon\lambda}{2} |X_{r}|,$$
		which contradicts that $G$ is a $\lambda$-expander and Claim \ref{claim:eps_exp}. Therefore, expansion must have happened for some $0 \leq i \leq r-1$, finishing the proof.
	\end{proof}

	Let $G$ be an $r$-line-graph such that $V(G)\subset A_{1}\times \dots\times A_{r}$, and let $\sigma\in S_r$. Say that a sequence $a_{1},\dots,a_{rk}$ of distinct elements of $A_{1}\cup \dots\cup A_{r}$ is a \emph{$\sigma$-path} in $G$ if 
	\begin{enumerate}
		\item $a_{i}\in A_{\sigma(j)}$, where $j\equiv i \pmod{r}$,
		\item $(a_{i},\dots,a_{i+r-1})\in V(G)$ for $i=1,\dots,rk-r+1$.
	\end{enumerate}
	Note that the sequence $a_{1},\dots,a_{rk}$ corresponds to a tight path in the hypergraph identified with $G$.
	
	Also, if $x_{1},\dots,x_{k}$ is a sequence of vertices of $G$, say that $x_{1},\dots,x_{k}$ forms a \emph{$\sigma$-path} if $a_{1},\dots,a_{rk}$ is a $\sigma$-path, where $a_{ri-r+1},\dots,a_{ri}$ are the coordinates of $x_{i}$ (in the order given by $\sigma$). Also, if $x,y\in V(G)$, say that \emph{$y$ can be reached from $x$ by a $\sigma$-path}, if there exists a $\sigma$-path $x_{1},\dots,x_{k}$ such that $x_{1}=x$ and $x_{k}=y$. Note that the statement that $y$ can be reached from $x$ by a $\sigma$-path is equivalent to the statement that $x$ can be reached from $y$ by a $\tau$-path, where $\tau$ is the reverse of $\sigma$ (that is, $\tau(i)=\sigma(r+1-i)$ for $i\in [r]$). The \emph{size} of a $\sigma$-path $x_{1},\dots,x_{k}$ is $k$.
	
	\begin{lemma}\label{lemma:paths}
		Let $\sigma\in S_{r}$, let $\epsilon,\lambda>0$ and let $n,d$ be positive integers such that $500r^4\log n<\epsilon^2\lambda^2 d$. Let $G$ be an $r$-line-graph on $n$ vertices that is a $(\lambda,d)$-expander, and let $x\in V(G)$. Then at least $(1-\epsilon)n$ vertices of $G$ can be reached by a $\sigma$-path of size at most $\frac{5r\log n}{\epsilon\lambda}$ from $x$.
	\end{lemma}
	
	\begin{proof}
		Suppose that $V(G)\subset A_1\times \dots\times A_r$. Let $X_{1}=\{x\}$ and let $X_{i}$ be the set of vertices that can be reached from $x$ by a $\sigma$-path of size $i$. For $y\in X_{i}$, let $x=x_1,\dots,x_i=y$ be a $\sigma$-path, and let $F(y)\subset A_1\cup \dots\cup A_r$ be the set of all coordinates appearing in $x_1,\dots,x_{i-1}$.  For $y\in V(G)\setminus X_{i}$, let $F(y)=\emptyset$, and set  $\mathcal{F}=(F(y))_{y\in V(G)}$. Then $\partial^{\sigma}(X_i,\mathcal{F})\subset X_{i+1}$. Note that $|F(y)|<ri$ for every $y\in V(G)$. If $|X_{i}|\leq (1-\epsilon)n$ and $ri\leq \frac{\epsilon\lambda d}{100r^2}$, then we can apply Lemma \ref{lemma:sigma_expansion} to get
		$$|X_{i+1}|\geq |\partial^{\sigma}(X_i,\mathcal{F})|\geq \left(1+\frac{\epsilon\lambda}{4r}\right)|X_i|.$$
		Hence, by induction on $i$, if $i<\frac{\epsilon\lambda d}{100r^3}$, then either   $|X_{i}|\geq \left(1+\frac{\epsilon\lambda}{4r}\right)^{i}$ or $|X_j|\geq (1-\epsilon)n$ for some $j\leq i$. Setting 
		$I:=\frac{5r\log n}{\epsilon\lambda}< \frac{\epsilon\lambda d}{100r^3}$, we have $\left(1+\frac{\epsilon\lambda}{4r}\right)^{I}>n$, which implies $|X_{j}|\geq (1-\epsilon)n$ for some $j\leq I$. 
	\end{proof}

\noindent
	\textbf{Remark.} Following the same proof, it is not hard to show the following strengthening of Lemma \ref{lemma:paths}. Let $L$ be a positive integer such that $\frac{5r\log n}{\epsilon\lambda}<L<\frac{\epsilon\lambda d}{100r^3}$. Then at least $(1-\epsilon)n$ vertices of $G$ can be reached by a $\sigma$-path of size  exactly $L$ from $x$.
	
	\section{Finding tight paths}

	This section contains the bulk of the proof of Theorem \ref{thm:mainthm2}. Here, we prove that if $G$ is an $r$-line-graph with good expansion properties and $\sigma\in S_r$, then either there exists a $\sigma$-path from any vertex $x$ to any other vertex $y$, or $G$ contains a small subgraph with unusually high density. Let us give a rough outline of the proof.
	
	 Suppose that $V(G)\subset A_1\times\dots\times A_r$, then we partition each of the sets $A_1,\dots,A_r$ into two parts, which then gives a partition of $V(G)$ into $2^r$ parts. We find a partition such that $x$ and $y$ are in different parts $G_1$ and $G_2$, no vertex in $V(G_1)$ shares a coordinate with a vertex in $V(G_2)$, and the vertices of $V(G)$ are uniformly distributed among the $2^r$ parts. Let $\tau$ be the reverse of $\sigma$, that is $\tau(i)=\sigma(r+1-i)$ for $i\in [r]$.  Then our goal is to find two vertices $z\in V(G_1)$ and $z'\in V(G_2)$ such that $z'$ is a $\sigma$-neighbor of $z$, there is a $\sigma$-path $P$ from $x$ to $z$ in $V(G_1)$, and there is a $\tau$-path $P'$ from $y$ to $z'$ in $V(G_2)$. Indeed, then $P\cup P'$ is a $\sigma$-path from $x$ to $y$.
	
	Unfortunately, we are not quite able to achieve this. One of the main difficulties is that while $G$ might have good expansion properties, this might not be true for $G_1$ or $G_2$. Instead, for $i=1,2$, we cover most vertices of $G_i$ with expander subgraphs $G_{i,1},\dots,G_{i,k_i}$ using Lemma \ref{lemma:expander_covering}, and argue that either the number of expanders used is small, or one of the expander subgraphs is small. In the latter case, we found our small subgraph of $G$ with unusually high density. 
	
	Hence, suppose that both $k_1$ and $k_2$ are small. In this case,  we will choose a vertex $x_j \in G_{1,j}$ and a $\sigma$-path $P_{1,j}$ connecting $x$ with $x_j$ in $G$. Since $G$ is expander this is possible to do for most indices $j$.
	Let $U_1$ be the set of coordinates appearing on the union of paths $P_{1,j}$. Similarly, for most indices $j$, we will choose a vertex $y_j \in G_{2,j}$ and a $\tau$-path $P_{2,j}$ connecting $y$ with $y_j$ such that the vertices of $P_{2,j}$ have no coordinates in $U_1$. Let $U_2$ be the set of coordinates appearing on the union of paths $P_{2,j}$. Using the paths $P_{i,j}$ and expansion properties of $G_{i,j}$	we show that most vertices of $G_1$ can be reached from $x$ by a $\sigma$-path, whose every vertex $z$ satisfies that $z$ has no coordinate in $U_2$, and $z$ is either in $V(G_1)$, or its coordinates are in $U_1$. Also, most vertices of $G_2$ can be reached from $y$ by a $\tau$-path, whose every vertex $z'$ satisfies that $z'$ has no coordinate in $U_1$, and $z'$ is either in $V(G_2)$, or its coordinates are in $U_2$. Then, we find two vertices $z\in V(G_1)$ and $z'\in V(G_2)$ such that $z'$ is a $\sigma$-neighbor of $z$, there is a $\sigma$-path $P$ from $x$ to $z$ satisfying the above properties, and there is a $\tau$-path $P'$ from $y$ to $z'$ satisfying the above properties. Then $P\cup P'$ is a $\sigma$-path from $x$ to $y$.
	
	In the next claim, we show that if we randomly partition the sets $A_1,\dots,A_r$, then the vertices are well distributed among the $2^r$ parts.
	
	\begin{claim}\label{claim:partition}
		Let $\epsilon>0$, then there exists $c=c(r,\epsilon)>0$ such that the following holds. Let $G$ be an $r$-line-graph on $n$ vertices such that $V(G)\subset A_1\times\dots \times A_r$ and the minimum degree of $G$ is at least $c(\log n)$. Also, let $x,y\in V(G)$ such that $x$ and $y$ share no coordinates. Then for $i\in [r]$, there exists a partition of $A_{i}$ into two sets, $A_{i,1}$ and $A_{i,2}$, with the following three properties:
		\begin{enumerate}
		    \item $x\in A_{1,1}\times\dots \times A_{r,1}$ and $y\in A_{1,2}\times\dots \times A_{r,2}$.
		    \item Given vector $\e=(e_1,\dots,e_r)$ let $A_{\e}=A_{1,e_1}\times\dots\times A_{r,e_r}$. Then for all $\e \in \{1,2\}^r$ and every block $B$ of $G$, either $B\cap A_{\e}=\emptyset$, or
			$$\frac{1-\epsilon/2r}{2}|B|\leq |B \cap A_{\e}|\leq \frac{1+\epsilon/2r}{2}|B|.$$
			\item For every $\e \in \{1,2\}^r$, $$\frac{1-\epsilon}{2^r}n\leq |V(G)\cap A_{\e}|\leq \frac{1+\epsilon}{2^r}n.$$
		\end{enumerate}
	\end{claim}
\begin{proof}
	 For $i\in [r]$, partition $A_i$ randomly into two sets $A_{i,1}$ and $A_{i,2}$ such that $x\in A_{1,1}\times \dots\times A_{r,1}$ and $y\in A_{1,2}\times \dots\times A_{r,2}$. More precisely, each element of $A_i\setminus\{x(i),y(i)\}$ is in either $A_{i,1}$ or $A_{i,2}$ independently with probability $1/2$. This partition clearly satisfies property 1 and we prove that for large enough $c$, with positive probability, it satisfies also 2 and 3. 
	 
	 Let $B$ be an $i$-block and let $\e \in \{1,2\}^r$. Without loss of generality we can assume that $i=1$. Then for $i=2,\dots,r$, there exists $a_{i}\in A_{i}$ such that $B\subset A_{1}\times \{a_2\}\times\dots\times \{a_r\}$. Therefore, if $a_{i}\not\in A_{i,e_i}$ for some $i\in \{2,\dots,r\}$, then $B\cap A_{\e}=\emptyset$. Otherwise, each element of $B'=B\setminus \{x,y\}$ appears in $A_{\e}$ independently with probability $1/2$. Therefore, choosing $c$ large enough and writing $m=|B'|$, we have by Chernoff's inequality
	 $$\mathbb{P}\left(|B'\cap A_e|-\frac{m}{2}\geq \frac{\epsilon m}{6r}\right)\leq 2e^{-\frac{\epsilon^2 m}{72}} \leq 2n^{-\frac{\epsilon^2c}{72r^2}}<\frac{1}{2^{r+1}r n}.$$
	 Here, $|B|\geq m\geq |B|-2$ and $\frac{\epsilon m}{12r}\geq 2$, so with probability at least $1-\frac{1}{2^{r+1}r n}$, we also have $\frac{1-\epsilon/(2r)}{2}|B|\leq |B\cap A_e|\leq \frac{1+\epsilon/(2r)}{2}|B|$. 
	 Since the number of blocks $B$ of $G$ is at most $rn$, and there are $2^r$ choices for $A_{\e}$, by the union bound the probability that property 2 holds is at least $\frac{1}{2}$.

	 To complete the proof we show that property 2 implies 3. For $i\in [r]$, let $A_{i,0}=A_{i}$, and for $\e \in \{0,1,2\}^{r}$, as before $A_{\e}=A_{1,e_1}\times\dots\times A_{r,e_r}$. We show, by induction, that if $\e \in \{0,1,2\}^{r}$ is a vector with exactly $s$ nonzero coordinates, then
	 $$\frac{(1-\frac{\epsilon}{2r})^{s}}{2^s}n\leq |V(G)\cap A_{\e}|\leq \frac{(1+\frac{\epsilon}{2r})^{s}}{2^s}n.$$
	 When $s=r$, this implies property 3. For $s=0$, the statement is trivially true, so let us suppose that $s\geq 1$, and that the statement holds for $s-1$ instead of $s$. Let $\e \in \{0,1,2\}^{r}$ such that $e$ contains $s$ nonzero coordinates, and suppose  that $e_{i}\neq 0$. Let $\f \in \{0,1,2\}^r$ be the vector we get after changing $e_{i}$ to 0 in $\e$. Note that if $B$ is an $i$-block of $G$, then $B$ is either disjoint from or completely contained in $A_{\f}$. Also, for each $B$ contained in $A_{\f}$, we can change uniquely the zero coordinates of $\f$ to be either $1$ or $2$ to obtain a vector $\g \in \{1,2\}^r$ such that $B\cap A_{\g}=B\cap A_{\e}\neq \emptyset$. Then by property 2, we have $\frac{1-\epsilon/(2r)}{2}|B|\leq |B\cap A_{\e}|\leq \frac{1+\epsilon/(2r)}{2}|B|.$ As this holds for every $i$-block $B$ contained in $A_{\f}$, we have $\frac{1-\epsilon/(2r)}{2}|V(G)\cap A_{\f}|\leq |V(G)\cap A_{\e}|\leq \frac{1+\epsilon/(2r)}{2}|V(G)\cap A_{\f}|$.
	 Since, by induction, $\frac{(1-\epsilon/(2r))^{s-1}}{2^{s-1}}n\leq |V(G)\cap A_{\f}|\leq \frac{(1+\epsilon/(2r))^{s-1}}{2^{s-1}}n$ this implies $\frac{(1-\epsilon/(2r))^{s}}{2^s}n\leq |V(G)\cap A_{\e}|\leq \frac{(1+\epsilon/(2r))^{s}}{2^s}n$.
\end{proof}

Now we are ready to prove the main lemma of this section.
	
	\begin{lemma}\label{lemma:density}
		There exist $c_1,c_2,c_3,c_4>0$ depending only on $r$ such that the following holds. Let $\lambda>0$, $K>1$, and let $n,d$ be positive integers such that $K<\frac{c_1\lambda^2 d}{\log n}$, $\lambda\leq \frac{1}{2\log_2 n}$ and $d\geq \frac{c_2\log n}{\lambda^{2}}$. Let $G$ be an $r$-line-graph with $n$ vertices that is $(\lambda,d)$-expander, and let $x,y\in V(G)$ such that $x$ and $y$ share no coordinates. Then either $G$ contains a $\sigma$-path of size at most $\frac{c_{3}\log n}{\lambda}$ from $x$ to $y$, or $G$ has a subgraph with at most  $\frac{n}{K}$ vertices and minimum degree at least $c_4d$.
	\end{lemma}
	
	\begin{proof}
		Let $\epsilon=2^{-r-6}$, and without loss of generality, let $\sigma$ be the permutation $12\dots r$. Also, suppose that $n\geq 2^{-\epsilon}$, otherwise, we can choose $c_1$ small enough to guarantee that $1<K\leq\frac{c_1\lambda^2 d}{\log n}$ cannot be satisfied. This also implies that $\lambda<\epsilon$. Let $c=c(r,\epsilon)$ be the constant given by Claim \ref{claim:partition}. We show that the constants $c_1=\frac{\epsilon^2}{240r^{4}}$, $c_2=\max\{10^5r^5\epsilon^{-3},c\}$, $c_3=40r\epsilon^{-1}$ and $c_4=\frac{\epsilon}{6r}$ suffice.
		
		Suppose that $G$ contains no subgraph with at most $\frac{n}{K}$ vertices and minimum degree at least $c_4d$. 
		Let $V(G)\subset A_1\times \dots\times A_r$, and for $i\in [r]$, partition $A_i$ into two sets $A_{i,1}$ and $A_{i,2}$ satisfying Claim \ref{claim:partition}. This can be done since $d\geq \frac{c_2\log n}{\lambda^2}\geq c\log n$.
		For $\e \in \{1,2\}^{r}$, recall that $A_{\e}=A_{1,e_1}\times\dots\times A_{r,e_r}$ and let $G_{\e}$ be the subgraph of $G$ induced on $A_{\e} \cap V(G)$. For simplicity, write $G_1$ and $G_2$ instead of $G_{(1,\dots,1)}$ and $G_{(2,\dots,2)}$, respectively. By Claim \ref{claim:partition}, we have the following properties:
		\begin{enumerate}
			\item $x\in V(G_1)$ and $y\in V(G_2)$,
			\item for  every $\e \in\{1,2\}^r$, $\frac{1+\epsilon}{2^r}n\geq |V(G_{\e})|\geq \frac{1-\epsilon}{2^r}n$, and
			\item for every block $B$ of $G$, and every $\e \in\{1,2\}^r$, either $B\cap V(G_{\e})=\emptyset$, or $$\frac{1+\frac{\epsilon}{2r}}{2}|B|\geq |B\cap V(G_{\e})|\geq \frac{1-\frac{\epsilon}{2r}}{2}|B|\geq \frac{d}{3}.$$
		\end{enumerate}
		
		Let $j\in\{1,2\}$. As the density of $G_{j}$ is at least $d/3$, we can apply Lemma \ref{lemma:expander_covering} to find vertex disjoint subgraphs $G_{j,1},\dots,G_{j,k_j}$ of $G_{j}$ such that $G_{j,i}$ is a $(\lambda,\frac{\epsilon d}{6r})$-expander, and $G_{j,1},\dots,G_{j,k_j}$ cover at least $(1-\epsilon)|V(G_{j})|$ vertices of $G_{j}$. By our assumption on the nonexistence of subgraphs of size at most $\frac{n}{K}$ and minimum degree $c_{4}d=\frac{\epsilon d}{6r}$, the size of each $G_{j,i}$ is at least $\frac{n}{K}$, which implies that $k_j\leq K$. 
		
		Let $G'$ be the graph we get after removing the coordinates of $y$ from $G$. As $r\leq \frac{\lambda d}{4r}$, we can apply Lemma \ref{lemma:expander_robust} to conclude that $G'$ is a $(\frac{\lambda}{2},\frac{d}{2})$-expander on at least $(1-\frac{r}{d})|V(G)|\geq (1-\epsilon)|V(G)|$ vertices. Let $X\subset V(G')$ be the set of vertices $z$ such that $z$ can be reached from $x$ by a $\sigma$-path of size at most $L:=\frac{5r\log n}{\epsilon(\lambda/2)}$ in $G'$. Noting that $500r^4\log n\leq \epsilon^{2}(\frac{\lambda}{2})^{2}\frac{d}{2}$ holds by the choice of $c_2$, we can apply Lemma \ref{lemma:paths} to get $|X|\geq (1-\epsilon)|V(G')|\geq (1-2\epsilon)|V(G)|$.
		
		For $i=1,\dots,k_1$, if $X\cap V(G_{1,i})$ is nonempty, pick an arbitrary vertex $x_i\in X\cap V(G_{1,i})$, which we call the representative of $G_{1,i}$. Without loss of generality, let $1,\dots,\ell_1$ be the indices $i$ for which $G_{1,i}$ has a representative, then $\sum_{i=\ell_1+1}^{k_1}|V(G_{1,i})|\leq |V(G)|-|X|\leq 2\epsilon|V(G)|$. Therefore, 
		\begin{equation}\label{equ:total_size}
		\sum_{i=1}^{\ell_1}|V(G_{1,i})|\geq (1-\epsilon)|V(G_1)|-2\epsilon |V(G)|\geq (1-2^{r+2}\epsilon)|V(G_1)|,
		\end{equation}
		where the last inequality holds by the bound on $|V(G_1)|$ from property 2.
		
		For $i=1,\dots,\ell_1$, let $P_{1,i}\subset V(G)$ be a $\sigma$-path of size at most $L$ from $x$ to $x_i$, and let $P_1=\bigcup_{i=1}^{\ell_1}P_{1,i}$. Also, let $U_1$ be the set of coordinates appearing in the vertices of $P_1$. Then $|P_1|\leq LK$, and $|U_1|\leq rLK\leq \frac{\lambda d}{4r}$, where the last inequality holds by the choice of $c_1$. Let $G''$ be the subgraph of $G$ after the removal of the elements of $U_1$.  We can apply Lemma \ref{lemma:expander_robust} to conclude that $G''$ is a $(\frac{\lambda}{2},\frac{d}{2})$-expander on at least $(1-\frac{|U_1|}{d})|V(G)|$ vertices. Here, $(1-\frac{|U_1|}{d})|V(G)|\geq (1-\frac{\lambda}{4r})|V(G)|>(1-\epsilon)|V(G)|$ holds.
		
		Let $\tau$ be the reverse of $\sigma$, that is, the permutation $r(r-1)\dots1$. Let $Y$ be the set of vertices in $G''$ which can be reached from $y$ by $\tau$-path of size at most $L$ in $G''$. Then $|Y|\geq (1-\epsilon)|V(G'')|>(1-2\epsilon)|V(G)|$ by Lemma \ref{lemma:paths}. For $i=1,\dots,k_2$, if $Y\cap V(G_{2,i})$ is nonempty, pick an arbitrary vertex $y_i\in Y\cap V(G_{2,i})$, which we call the representative of $G_{2,i}$. Without loss of generality, let $1,\dots,\ell_2$ be the indices $i$ for which $G_{2,i}$ has a representative, then $\sum_{i=\ell_2+1}^{k_2}|V(G_{2,i})|\leq |V(G)|-|Y|\leq 2\epsilon|V(G)|$. Therefore, $$\sum_{i=1}^{\ell_2}|V(G_{2,i})|\geq (1-\epsilon)|V(G_2)|-2\epsilon |V(G)|\geq (1-2^{r+2}\epsilon)|V(G_2)|.$$
		
		For $i=1,\dots,\ell_2$, let $P_{2,i}\subset V(G)$ be a $\tau$-path of size at most $L$ from $y$ to $y_i$, and let $P_2=\bigcup_{i=1}^{\ell_2}P_{2,i}$. Also, let $U_2$ be the set of coordinates appearing in the vertices of $P_2$. Then $|P_2|\leq LK$ and $|U_2|\leq rLK$.
		
		For $j=1,2$ and $i=1,\dots,\ell_j$, let $H_{j,i}$ be the graph we get after removing every element of $U=U_{1}\cup U_{2}$ from $G_{j,i}$, with the exception of the coordinates of $x_i$ in case $j=1$, and with the exception of the coordinates of $y_i$ in case $j=2$. Here, $|U|\leq 2rLK<\frac{\lambda(\epsilon d/(6r))}{4r}$ holds by the choice of $c_1$, so we can apply Lemma \ref{lemma:expander_robust} to conclude that $H_{j,i}$ is a $(\frac{\lambda}{2},\frac{\epsilon d}{12r})$-expander on at least $$\left(1-\frac{|U|}{\epsilon d/(6r)}\right)|V(G_{j,i})|>\left(1-\frac{\lambda}{4r}\right)|V(G_{j,i})|> (1-\epsilon)|V(G_{j,i})|$$ vertices. Let $X_{i}$ be the set of vertices in $H_{1,i}$ that can be reached from $x_i$ by $\sigma$-path of size at most $r$ in $H_{1,i}$. Also, let $Y_{i}$ be the set of vertices in $H_{2,i}$ that can be reached from $y_i$ by $\tau$-path of size at most $L$ in $H_{2,i}$. See Figure \ref{figure} for an illustration. Noting that $500r^4\log n<\epsilon^{2}(\frac{\lambda}{2})^2(\frac{\epsilon d}{12r})$ holds by the choice of $c_2$, we can apply Lemma \ref{lemma:paths} to deduce that $|X_{i}|\geq (1-\epsilon)|V(H_{1,i})|\geq (1-2\epsilon)|V(G_{1,i})|$, and similarly $|Y_{i}|\geq (1-2\epsilon)|V(G_{2,i})|$.

	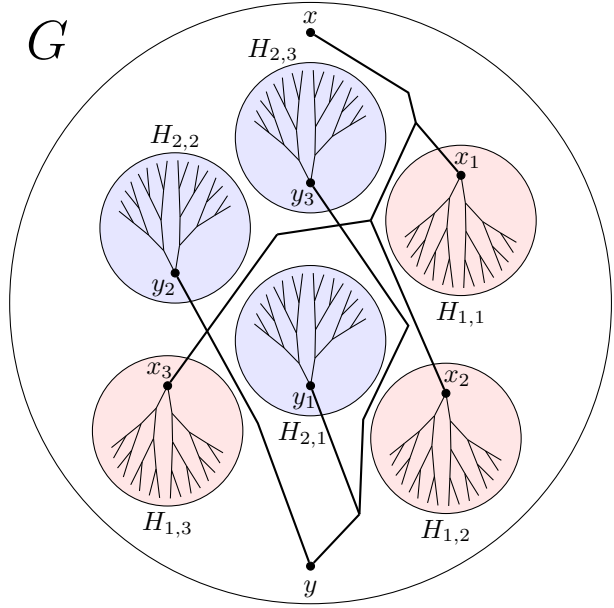
\begin{figure}
	\centering
	\begin{tikzpicture}
	\draw (0,0) circle (4) ;
	\node at (-3.5,3.5) {\huge $G$};
	
	\draw[fill=blue!10!white] (0,-0.5) circle (1);
	\draw[fill=red!10!white] (2,1.1) circle (1);
	\draw[fill=blue!10!white] (-1.8,1) circle (1);
	\draw[fill=red!10!white] (1.8,-1.8) circle (1);
	\draw[fill=red!10!white] (-1.9,-1.7) circle (1);
	\draw[fill=blue!10!white] (0,2.2) circle (1);
	\node[vertex] (x) at (0,3.6) {}; \node at (0,3.8) {\small $x$};
	\node[vertex] (y) at (0,-3.5) {};\node at (0,-3.8) {\small $y$};
	\node[vertex] (x1) at (2,1.7) {};\node at (2.09,1.9) {\small $x_1$};  \node at (2,-0.15) {\small $H_{1,1}$};
	\node[vertex] (x2) at (1.8,-1.2) {};\node at (1.95,-1) {\small $x_2$}; \node at (1.8,-3.05) {\small $H_{1,2}$};
	\node[vertex] (x3) at (-1.9,-1.1) {};\node at (-2,-0.9) {\small $x_3$};  \node at (-1.9,-2.95) {\small $H_{1,3}$};
	
	\node[vertex] (y1) at (0,-1.1) {};\node at (-0.08,-1.3) {\small $y_1$}; \node at (-0.1,-1.75) {\small $H_{2,1}$};
	\node[vertex] (y2) at (-1.8,0.4) {};\node at (-1.95,0.2) {\small $y_2$}; \node at (-1.8,2.2) {\small $H_{2,2}$};
	\node[vertex] (y3) at (0,1.6) {};\node at (-0.1,1.4) {\small $y_3$};  \node at (-0.5,3.35) {\small $H_{2,3}$};
	
	\draw[thick] (x) -- (1.3,2.8) -- (1.4,2.4) -- (x1);
	\draw[thick] (1.4,2.4) -- (0.8,1.1);
	\draw[thick] (x2) -- (0.8,1.1);
	\draw[thick] (x3) -- (-0.443739,0.913236);
	\draw[thick] (0.8,1.1) -- (-0.443739,0.913236);
	
	\draw[thick] (y2) -- (-0.7,-1.6) -- (y);
	\draw[thick] (y1) -- (0.65,-2.81);
	\draw[thick] (y3) -- (1.3,-0.3) -- (0.7,-1.55)  -- (0.65,-2.81); 
	\draw[thick] (0.65,-2.81) -- (y);
	
	\draw (0,1.6) -- (0.0580543,1.96579);
	\draw (0,1.6) -- (-0.15155,1.8979);
	\draw (0.0580543,1.96579) -- (0.287352,2.21594);
	\draw (0.0580543,1.96579) -- (0.0653794,2.34426);
	\draw (-0.15155,1.8979) -- (-0.187155,2.29996);
	\draw (-0.15155,1.8979) -- (-0.352203,2.10376);
	\draw (0.287352,2.21594) -- (0.524554,2.3656);
	\draw (0.287352,2.21594) -- (0.415834,2.5413);
	\draw (0.0653794,2.34426) -- (0.251011,2.66589);
	\draw (0.0653794,2.34426) -- (0.0523207,2.72256);
	\draw (-0.187155,2.29996) -- (-0.153429,2.70367);
	\draw (-0.187155,2.29996) -- (-0.338477,2.61177);
	\draw (-0.352203,2.10376) -- (-0.477858,2.45925);
	\draw (-0.352203,2.10376) -- (-0.552767,2.26669);
	\draw (0.524554,2.3656) -- (0.73646,2.49187);
	\draw (0.524554,2.3656) -- (0.69633,2.62861);
	\draw (0.415834,2.5413) -- (0.631058,2.7553);
	\draw (0.415834,2.5413) -- (0.543,2.86735);
	\draw (0.251011,2.66589) -- (0.435337,2.96072);
	\draw (0.251011,2.66589) -- (0.311956,3.03204);
	\draw (0.0523207,2.72256) -- (0.177312,3.07874);
	\draw (0.0523207,2.72256) -- (0.0362648,3.09912);
	\draw (-0.153429,2.70367) -- (-0.106091,3.09246);
	\draw (-0.153429,2.70367) -- (-0.244617,3.05899);
	\draw (-0.338477,2.61177) -- (-0.37431,2.99992);
	\draw (-0.338477,2.61177) -- (-0.490489,2.91738);
	\draw (-0.477858,2.45925) -- (-0.588958,2.81436);
	\draw (-0.477858,2.45925) -- (-0.666162,2.69457);
	\draw (-0.552767,2.26669) -- (-0.719313,2.56234);
	\draw (-0.552767,2.26669) -- (-0.746493,2.42244);
	
	\draw (-1.8,0.4) -- (-1.74195,0.765786);
	\draw (-1.8,0.4) -- (-1.95155,0.697903);
	\draw (-1.74195,0.765786) -- (-1.51265,1.01594);
	\draw (-1.74195,0.765786) -- (-1.73462,1.14426);
	\draw (-1.95155,0.697903) -- (-1.98716,1.09996);
	\draw (-1.95155,0.697903) -- (-2.1522,0.903756);
	\draw (-1.51265,1.01594) -- (-1.27545,1.1656);
	\draw (-1.51265,1.01594) -- (-1.38417,1.3413);
	\draw (-1.73462,1.14426) -- (-1.54899,1.46589);
	\draw (-1.73462,1.14426) -- (-1.74768,1.52256);
	\draw (-1.98716,1.09996) -- (-1.95343,1.50367);
	\draw (-1.98716,1.09996) -- (-2.13848,1.41177);
	\draw (-2.1522,0.903756) -- (-2.27786,1.25925);
	\draw (-2.1522,0.903756) -- (-2.35277,1.06669);
	\draw (-1.27545,1.1656) -- (-1.06354,1.29187);
	\draw (-1.27545,1.1656) -- (-1.10367,1.42861);
	\draw (-1.38417,1.3413) -- (-1.16894,1.5553);
	\draw (-1.38417,1.3413) -- (-1.257,1.66735);
	\draw (-1.54899,1.46589) -- (-1.36466,1.76072);
	\draw (-1.54899,1.46589) -- (-1.48804,1.83204);
	\draw (-1.74768,1.52256) -- (-1.62269,1.87874);
	\draw (-1.74768,1.52256) -- (-1.76374,1.89912);
	\draw (-1.95343,1.50367) -- (-1.90609,1.89246);
	\draw (-1.95343,1.50367) -- (-2.04462,1.85899);
	\draw (-2.13848,1.41177) -- (-2.17431,1.79992);
	\draw (-2.13848,1.41177) -- (-2.29049,1.71738);
	\draw (-2.27786,1.25925) -- (-2.38896,1.61436);
	\draw (-2.27786,1.25925) -- (-2.46616,1.49457);
	\draw (-2.35277,1.06669) -- (-2.51931,1.36234);
	\draw (-2.35277,1.06669) -- (-2.54649,1.22244);
	
	\draw (0,-1.1) -- (0.0580543,-0.734214);
	\draw (0,-1.1) -- (-0.15155,-0.802097);
	\draw (0.0580543,-0.734214) -- (0.287352,-0.484056);
	\draw (0.0580543,-0.734214) -- (0.0653794,-0.355743);
	\draw (-0.15155,-0.802097) -- (-0.187155,-0.400042);
	\draw (-0.15155,-0.802097) -- (-0.352203,-0.596244);
	\draw (0.287352,-0.484056) -- (0.524554,-0.3344);
	\draw (0.287352,-0.484056) -- (0.415834,-0.158702);
	\draw (0.0653794,-0.355743) -- (0.251011,-0.0341117);
	\draw (0.0653794,-0.355743) -- (0.0523207,0.0225614);
	\draw (-0.187155,-0.400042) -- (-0.153429,0.00367093);
	\draw (-0.187155,-0.400042) -- (-0.338477,-0.0882345);
	\draw (-0.352203,-0.596244) -- (-0.477858,-0.240755);
	\draw (-0.352203,-0.596244) -- (-0.552767,-0.433312);
	\draw (0.524554,-0.3344) -- (0.73646,-0.208133);
	\draw (0.524554,-0.3344) -- (0.69633,-0.0713876);
	\draw (0.415834,-0.158702) -- (0.631058,0.0552979);
	\draw (0.415834,-0.158702) -- (0.543,0.16735);
	\draw (0.251011,-0.0341117) -- (0.435337,0.260722);
	\draw (0.251011,-0.0341117) -- (0.311956,0.332043);
	\draw (0.0523207,0.0225614) -- (0.177312,0.378739);
	\draw (0.0523207,0.0225614) -- (0.0362648,0.399123);
	\draw (-0.153429,0.00367093) -- (-0.106091,0.392459);
	\draw (-0.153429,0.00367093) -- (-0.244617,0.358987);
	\draw (-0.338477,-0.0882345) -- (-0.37431,0.299917);
	\draw (-0.338477,-0.0882345) -- (-0.490489,0.217381);
	\draw (-0.477858,-0.240755) -- (-0.588958,0.114359);
	\draw (-0.477858,-0.240755) -- (-0.666162,-0.00542963);
	\draw (-0.552767,-0.433312) -- (-0.719313,-0.137659);
	\draw (-0.552767,-0.433312) -- (-0.746493,-0.277555);
	
	\draw (2,1.7) -- (2.05805,1.33421);
	\draw (2,1.7) -- (1.84845,1.4021);
	\draw (2.05805,1.33421) -- (2.28735,1.08406);
	\draw (2.05805,1.33421) -- (2.06538,0.955743);
	\draw (1.84845,1.4021) -- (1.81284,1.00004);
	\draw (1.84845,1.4021) -- (1.6478,1.19624);
	\draw (2.28735,1.08406) -- (2.52455,0.9344);
	\draw (2.28735,1.08406) -- (2.41583,0.758702);
	\draw (2.06538,0.955743) -- (2.25101,0.634112);
	\draw (2.06538,0.955743) -- (2.05232,0.577439);
	\draw (1.81284,1.00004) -- (1.84657,0.596329);
	\draw (1.81284,1.00004) -- (1.66152,0.688235);
	\draw (1.6478,1.19624) -- (1.52214,0.840755);
	\draw (1.6478,1.19624) -- (1.44723,1.03331);
	\draw (2.52455,0.9344) -- (2.73646,0.808133);
	\draw (2.52455,0.9344) -- (2.69633,0.671388);
	\draw (2.41583,0.758702) -- (2.63106,0.544702);
	\draw (2.41583,0.758702) -- (2.543,0.43265);
	\draw (2.25101,0.634112) -- (2.43534,0.339278);
	\draw (2.25101,0.634112) -- (2.31196,0.267957);
	\draw (2.05232,0.577439) -- (2.17731,0.221261);
	\draw (2.05232,0.577439) -- (2.03626,0.200877);
	\draw (1.84657,0.596329) -- (1.89391,0.207541);
	\draw (1.84657,0.596329) -- (1.75538,0.241013);
	\draw (1.66152,0.688235) -- (1.62569,0.300083);
	\draw (1.66152,0.688235) -- (1.50951,0.382619);
	\draw (1.52214,0.840755) -- (1.41104,0.485641);
	\draw (1.52214,0.840755) -- (1.33384,0.60543);
	\draw (1.44723,1.03331) -- (1.28069,0.737659);
	\draw (1.44723,1.03331) -- (1.25351,0.877555);
	
	\draw (1.8,-1.2) -- (1.85805,-1.56579);
	\draw (1.8,-1.2) -- (1.64845,-1.4979);
	\draw (1.85805,-1.56579) -- (2.08735,-1.81594);
	\draw (1.85805,-1.56579) -- (1.86538,-1.94426);
	\draw (1.64845,-1.4979) -- (1.61284,-1.89996);
	\draw (1.64845,-1.4979) -- (1.4478,-1.70376);
	\draw (2.08735,-1.81594) -- (2.32455,-1.9656);
	\draw (2.08735,-1.81594) -- (2.21583,-2.1413);
	\draw (1.86538,-1.94426) -- (2.05101,-2.26589);
	\draw (1.86538,-1.94426) -- (1.85232,-2.32256);
	\draw (1.61284,-1.89996) -- (1.64657,-2.30367);
	\draw (1.61284,-1.89996) -- (1.46152,-2.21177);
	\draw (1.4478,-1.70376) -- (1.32214,-2.05925);
	\draw (1.4478,-1.70376) -- (1.24723,-1.86669);
	\draw (2.32455,-1.9656) -- (2.53646,-2.09187);
	\draw (2.32455,-1.9656) -- (2.49633,-2.22861);
	\draw (2.21583,-2.1413) -- (2.43106,-2.3553);
	\draw (2.21583,-2.1413) -- (2.343,-2.46735);
	\draw (2.05101,-2.26589) -- (2.23534,-2.56072);
	\draw (2.05101,-2.26589) -- (2.11196,-2.63204);
	\draw (1.85232,-2.32256) -- (1.97731,-2.67874);
	\draw (1.85232,-2.32256) -- (1.83626,-2.69912);
	\draw (1.64657,-2.30367) -- (1.69391,-2.69246);
	\draw (1.64657,-2.30367) -- (1.55538,-2.65899);
	\draw (1.46152,-2.21177) -- (1.42569,-2.59992);
	\draw (1.46152,-2.21177) -- (1.30951,-2.51738);
	\draw (1.32214,-2.05925) -- (1.21104,-2.41436);
	\draw (1.32214,-2.05925) -- (1.13384,-2.29457);
	\draw (1.24723,-1.86669) -- (1.08069,-2.16234);
	\draw (1.24723,-1.86669) -- (1.05351,-2.02244);
	
	\draw (-1.9,-1.1) -- (-1.84195,-1.46579);
	\draw (-1.9,-1.1) -- (-2.05155,-1.3979);
	\draw (-1.84195,-1.46579) -- (-1.61265,-1.71594);
	\draw (-1.84195,-1.46579) -- (-1.83462,-1.84426);
	\draw (-2.05155,-1.3979) -- (-2.08716,-1.79996);
	\draw (-2.05155,-1.3979) -- (-2.2522,-1.60376);
	\draw (-1.61265,-1.71594) -- (-1.37545,-1.8656);
	\draw (-1.61265,-1.71594) -- (-1.48417,-2.0413);
	\draw (-1.83462,-1.84426) -- (-1.64899,-2.16589);
	\draw (-1.83462,-1.84426) -- (-1.84768,-2.22256);
	\draw (-2.08716,-1.79996) -- (-2.05343,-2.20367);
	\draw (-2.08716,-1.79996) -- (-2.23848,-2.11177);
	\draw (-2.2522,-1.60376) -- (-2.37786,-1.95925);
	\draw (-2.2522,-1.60376) -- (-2.45277,-1.76669);
	\draw (-1.37545,-1.8656) -- (-1.16354,-1.99187);
	\draw (-1.37545,-1.8656) -- (-1.20367,-2.12861);
	\draw (-1.48417,-2.0413) -- (-1.26894,-2.2553);
	\draw (-1.48417,-2.0413) -- (-1.357,-2.36735);
	\draw (-1.64899,-2.16589) -- (-1.46466,-2.46072);
	\draw (-1.64899,-2.16589) -- (-1.58804,-2.53204);
	\draw (-1.84768,-2.22256) -- (-1.72269,-2.57874);
	\draw (-1.84768,-2.22256) -- (-1.86374,-2.59912);
	\draw (-2.05343,-2.20367) -- (-2.00609,-2.59246);
	\draw (-2.05343,-2.20367) -- (-2.14462,-2.55899);
	\draw (-2.23848,-2.11177) -- (-2.27431,-2.49992);
	\draw (-2.23848,-2.11177) -- (-2.39049,-2.41738);
	\draw (-2.37786,-1.95925) -- (-2.48896,-2.31436);
	\draw (-2.37786,-1.95925) -- (-2.56616,-2.19457);
	\draw (-2.45277,-1.76669) -- (-2.61931,-2.06234);
	\draw (-2.45277,-1.76669) -- (-2.64649,-1.92244);
	\end{tikzpicture}
	\caption{An illustration of how we build $\sigma$-paths from $x$ and $\tau$-paths from $y$.}
	\label{figure}
\end{figure}

		Let $X'=\bigcup_{i=1}^{\ell_1}X_{i}$ and $Y'=\bigcup_{i=1}^{\ell_2}Y_{i}$, then $|X'|\geq \sum_{i=1}^{\ell_1}(1-2\epsilon)|V(G_{1,i})|\geq (1-2^{r+3}\epsilon)|V(G_{1})|$ by inequality (\ref{equ:total_size}), and similarly $|Y'|\geq (1-2^{r+3}\epsilon)|V(G_{2})|$.
		
		Here, $X'$ and $Y'$ have the following property. Every vertex $z\in X'$ can be reached from $x$ by a $\sigma$-path $P_{z}$ of size at most $2L$ such that every coordinate of every vertex of $P_{z}$ is in the set $(A_{1,1}\cup \dots\cup A_{1,r}\cup U_1)\setminus U_{2}$. Also,  every vertex $z'\in Y'$ can be reached from $y$ by a $\tau$-path $P_{z'}$ of size at most $2L$ such that every coordinate of every vertex of $P_{z'}$ is in the set $(A_{2,1}\cup \dots\cup A_{2,r}\cup U_2)\setminus U_{1}$. Therefore, if $z\in X'$ and $z'\in Y'$, then no vertex in $P_z$ shares a coordinate with any vertex in $P_{z'}$.
		
		In order to finish the proof, it is enough to find $z=(z_1,\dots,z_r)\in X'$ and $z'=(z_1',\dots,z_r')\in Y'$ such that $w_{i}=(z_1',\dots,z_i',z_{i+1},\dots,z_r)$ are vertices of $G$ for $i=1,\dots,r-1$. Indeed, then $P_z\cup P_{z'}$ is a $\sigma$-path of size at most $4L=\frac{c_{3}\log n}{\lambda}$ from $x$ to $y$. But this is equivalent to the statement that $\partial^{\sigma}_{G}(X')\cap Y'\neq \emptyset$. 
		
		\begin{claim}
			Let $W\subset |V(G_1)|$. Then $|\partial_{G}^{\sigma}(W)\cap V(G_2)|\geq (1-\epsilon)|W|$.
		\end{claim}
		\begin{proof}
			For $i\in \{0,\dots,r\}$, let $\e_{i}\in \{1,2\}^r$ be the vector whose first $i$ coordinates are 2, and the last $r-i$ coordinates are 1. Let $W_{0}=W$ and for $i=1,\dots,r$, let $W_i=\partial^{(i)}(W_{i-1})\cap V(G_{\e_i})$. Then $\partial_{G}^{\sigma}(W)\cap V(G_2)=W_r$. We show that $|W_i|\geq (1-\epsilon/r) |W_{i-1}|$, then we get that $|W_r|\geq (1-\epsilon/r)^r|W|\geq (1-\epsilon)|W|$, finishing the proof.
			
			 Let $\mathcal{B}$ be the set of $i$-blocks of $G$ having a nonempty intersection with $W_{i-1}$. Let $B\in \mathcal{B}$, then $B\cap V(G_{\e_i})=B\cap W_{i}$. But $|B\cap V(G_{\e_i})|\geq \frac{1-\epsilon/(2r)}{2}|B|$ and $|B\cap W_{i-1}|\leq |B\cap V(G_{\e_{i-1}})|\leq\frac{1+\epsilon/(2r)}{2}|B|$ by 2., so $|B\cap W_i|\geq \frac{1-\epsilon/(2r)}{1+\epsilon/(2r)}|B\cap W_{i-1}|\geq (1-\epsilon/r)|B\cap W_{i-1}|$. As this is true for every block in $\mathcal{B}$, we get $|W_{i}|\geq (1-\epsilon/r)|W_{i-1}|$.
		\end{proof}
		
		By the previous claim, we have $$|\partial^{\sigma}_{G}(X')\cap V(G_2)|\geq (1-\epsilon)|X'|\geq (1-2^{r+4}\epsilon)|V(G_{1})|>\frac{1}{2}|V(G_2)|$$
		where the third inequality holds by the bound on $|V(G_2)|$ from property 2. Since also $$|Y'|\geq (1-2^{r+3}\epsilon)|V(G_2)|>\frac{1}{2}|V(G_2)|,$$ we get that $\partial^{\sigma}_{G}(X')\cap Y'\neq \emptyset$, completing the proof.
	\end{proof}

\section{Finding a tight cycle}
	
The following statement follows easily from Lemma \ref{lemma:paths}.

	\begin{corollary}\label{cor:tight_cycle}
		There exist $c_1',c_2',c_3'>0$ depending only on $r$ such that the following holds. Let $K>1$ and let $n,d$ be positive integers such that $d\geq c_1'(\log n)^{3}$ and $K\leq c_2'\frac{d}{(\log n)^{3}}$. If $G$ is an $r$-line-graph on $n$ vertices of density at least $d$, then either $G$ contains a tight cycle, or $G$ contains a subgraph with minimum degree at least $c_3'd$ on at most $\frac{n}{K}$ vertices.
	\end{corollary}
	
	\begin{proof}
		Let $c_1,c_2,c_3,c_4$ be the constants given by Lemma \ref{lemma:density}. We show that $c_1'=\max\{64rc_{2},256r^3c_{3}\}$, $c_2'=\frac{c_1}{32r^2}$, $c_3'=\frac{c_{4}}{4r}$ suffices. Let $\lambda=\frac{1}{2\log_2 n}$. As $\mbox{dens}(G)\geq d$, $G$ contains a subgraph $H$ that is $(\lambda,\frac{d}{2r})$-expander, by Lemma \ref{lemma:expander}. Suppose that $H$ contains no subgraph with minimum degree at least $c_3'd$ on at most $\frac{|V(H)|}{K}\leq \frac{n}{K}$ vertices. Let $\sigma\in S_{r}$ be an arbitrary permutation.
		
		Let $x,y\in V(H)$ such that $x$ and $y$ share no coordinates. As the parameters $\lambda,\frac{d}{2r},K$ satisfy the desired conditions of Lemma \ref{lemma:density}, there exists a $\sigma$-path $P$ from $x$ to $y$ in $H$ of size at most $\frac{c_{3}\log n}{\lambda}< 4c_{3}(\log n)^2$. Let $U$ be the set of coordinates appearing in $P\setminus \{x,y\}$, and let $H'$ be the subgraph of $H$ we get after removing the elements of $U$. Note that $|U|\leq 16rc_{3}(\log n)^{2}\leq \frac{\lambda(d/(2r))}{4r}$, so we can apply Lemma \ref{lemma:expander_robust} to get that $H'$ is a $(\frac{\lambda}{2},\frac{d}{4r})$-expander. But then applying Lemma \ref{lemma:density} again, noting that $\frac{\lambda}{2},\frac{d}{4r},K$ also satisfy he desired conditions, we get that $H'$ contains a $\sigma$-path $P'$ from $y$ to $x$. Observe that $P\cup P'$ is a tight cycle, finishing the proof.
	\end{proof}

	\begin{proof}[Proof of Theorem \ref{thm:mainthm2}]
		Let $c_1',c_2',c_3'$ be the constants given by Corollary \ref{cor:tight_cycle}, and let $K=e^{(\log n)^{1/2}}$. Choose $c$ sufficiently large such that the following inequalities hold for every positive integer $n$: $c\geq 2\log(1/c_3')$,  $\exp(\frac{c}{2}(\log n)^{1/2})\geq c_{1}'(\log n)^{3}$ and $K(\log n)^3\leq c_2'\exp(\frac{c}{2}(\log n)^{1/2})$.
		Suppose that $G=G_{0}$ does not contain a tight cycle. Define the graphs $G_{1},G_{2},\dots,G_{k}$ for $k<\sqrt{\log n}$ with the following properties: 
		\begin{enumerate}
			\item for $i\geq 1$, $G_{i}$ is a subgraph of $G_{i-1}$, 
			\item $\mbox{dens}(G_{i})\geq (c_{3}')^{i}d$,
			\item $|V(G_{i})|\leq \frac{n}{K^{i}}$.
		\end{enumerate}
		Clearly, $G_{0}$ satisfies these properties. If $G_{i}$ is already defined satisfying these properties for some $0\leq i\leq \sqrt{\log n}$, define $G_{i+1}$ as follows. We have 
		$$d_{i}:=\mbox{dens}(G_{i})\geq (c_{3}')^{i+1}d\geq \exp\left(c(\log n)^{1/2}-\log\left(\frac{1}{c_{3}'}\right)(\log n)^{1/2}\right)\geq \exp\left(\frac{c}{2}(\log n)^{1/2}\right),$$
		which implies $d_{i}\geq c_{1}'(\log n)^{3}\geq c_{1}'(\log |V(G_{i})|)^{3}$ and $K\leq c_2'\frac{d_{i}}{(\log n)^{3}}\leq c_2'\frac{d_{i}}{(\log |V(G_i)|)^{3}}$. Therefore, we can apply Lemma \ref{cor:tight_cycle} to conclude that there exists a subgraph $G_{i+1}$ of $G_{i}$ with at most $\frac{|V(G_{i})|}{K}\leq \frac{n}{K^{i+1}}$ vertices and density at least  $(c_{3}')^{i+1}d$.
		
		However, this is a contradiction: for $I=\lfloor  \sqrt{\log n}\rfloor$, the graph $G_{I}$ has less vertices than its density. Therefore, $G$ must contain a tight cycle.
	\end{proof}

\section{Concluding remarks}
In this paper we proved that an $r$-uniform hypergraph $\mathcal{H}$ on $n$ vertices with $dn^2$ edges, where $d\geq e^{c\sqrt{\log n}}$, contains a tight cycle of length at most $O((\log n)^{2})$. 
Our proof can be also used to show that one can find a cycle of any specific length $L$, divisible by $r$, such that $\Omega((\log n)^{2})<L<de^{-O(\sqrt{\log n})}$. To achieve this, one needs just to follow the remark after Lemma \ref{lemma:paths} and its consequences. 

It is very plausible that our approach works also when $d$ is polylogarithmic in $n$. The only place in our paper that requires larger degree 	
comes from Lemma \ref{lemma:density}. It seems that after taking a random partition of the sets $A_1,\dots,A_r$, the graphs $G_1$ and $G_2$ should also have good expansion properties. If this is the case, then one can show that $\mbox{ex}(n,\mathcal{C}^{(r)})=n^{r-1}(\log n)^{O(1)}$. On the other hand, in order to prove $\ex(n,\mathcal{C}^{(r)})=O(n^{r-1})$, one seems to need new ideas.
	
Finally, let us mention a related conjecture of Conlon, see \cite{MPS}, about the extremal number of tight cycles of given constant length. Let $C^{(r)}_{\ell}$ denote the $r$-uniform tight cycle of length $\ell$.
\begin{conjecture}\label{conj:length}
There exists $c=c(r)>0$ such that for every $\ell\geq r+1$ which is divisible by $r$, we have $\mbox{ex}(n, C^{(r)}_{\ell})=O(n^{r-1+\frac{c}{\ell}})$.
\end{conjecture}

\noindent
Note that, it is essential that $r$ divides $\ell$. Otherwise a complete $r$-partite $r$-uniform hypergraph shows that the extremal number is $\Omega(n^{r})$.
For $r=3$, a solution of the above conjecture gives an improved upper bound on the maximum number of edges in a subgraph of the hypercube $\{0,1\}^n$ that contain no cycle of length $4k+2$ for large $k$.
See \cite{C10} for connection between these two problems. 

\vspace{0.3cm}
\noindent	
{\bf Acknowledgement.}\, We would like to thank Jacques Verstra\"ete for bringing \cite{C10} to our attention and Stefan Glock for useful discussions.

\end{document}